\documentclass[twoside,12pt]{article}
\usepackage{color}

\usepackage{graphics}
\usepackage{graphicx}

\usepackage{amssymb,amsfonts}
\usepackage{amsmath}
\usepackage{amsthm}

\def\hpp{${\cal HPP}$ }
\def\hpt{${\cal HPT}$ }

\def\ppHq{${\cal PP}_{\!4,q}$ }
\def\ppH45{${\cal PP}_{\!4,5}$ }
\def\H2R{$\mathbf{H}^2\!\times\!\mathbf{R}$\!\! }
\renewcommand{\vec}[1]{\mathbf{#1}}

\def\binomh#1#2{ \scalebox{.3}[1.2]{\textbf{)}}{\genfrac{}{}{0pt}{}{#1}{#2}}\scalebox{.3}[1.2]{\textbf{(}} }

\newtheorem{theorem}{Theorem}
\newtheorem{lemma}{Lemma}[section]

\newtheorem{remark}{Remark}

\title{\bf Pascal pyramid in the space $\mathbf{H}^2\!\times\!\mathbf{R}$
}
\author{ L\'aszl\'o N\'emeth\footnote{University of West Hungary,  Institute of Mathematics, Hungary. \textit{nemeth.laszlo@emk.nyme.hu}}}
\date{}

\begin{document}

\maketitle

\begin{abstract}
In this article we introduce a new type of Pascal pyramids.  
A regular squared mosaic in the hyperbolic plane yields a $(h^2r)$-cube mosaic in space $\mathbf{H}^2\!\times\!\mathbf{R}$ and the definition of the pyramid is based on this regular mosaic. The levels of the pyramid inherit some properties from the Euclidean and hyperbolic Pascal triangles. We give the growing method from level to level and show some illustrating figures. \\[1mm]
\end{abstract}

\section{Introduction}\label{sec:introduction} 

There are several approaches to generalize the Pascal's arithmetic triangle (see, for instance \cite{AhBel,Barry,BSz,SV}). A new type of variations of it is based on the hyperbolic regular mosaics denoted by Schl\"{a}fli's symbol $\{p,q\}$, where $(p-2)(q-2)>4$ (\cite{C}). Each regular mosaic induces a so-called hyperbolic Pascal triangle (see \cite{BNSz}), following and generalizing the connection between the classical Pascal's triangle and the Euclidean regular square mosaic $\{4,4\}$. For more details see \cite{BNSz,N_pyr,NSz_alter,NSz_recu}, but here we also collect some necessary information. We  use the attribute Pascal's (with apostrophe) only for the original, Euclidean arithmetic triangle and pyramid. 
 
The hyperbolic Pascal triangle based on the mosaic $\{p,q\}$ can be envisaged as a digraph, where the vertices and the edges are the vertices and the edges of a well defined part of  lattice $\{p,q\}$, respectively, and the vertices possess a value that give the number of different shortest paths from the base vertex to the given vertex. 
In this article we build on the hyperbolic squared mosaics, thus the other properties hold  just for mosaic $\{4,q\}$.  Figure~\ref{fig:Pascal_layer6} illustrates the hyperbolic Pascal triangle when $\{p,q\}=\{4,5\}$. 
Here the base vertex has two edges, the leftmost and the rightmost vertices have three, the others have $q$ edges. The quadrilateral shape cells surrounded by the appropriate edges correspond to the squares in the mosaic.
Apart from the winger elements, certain vertices (called ``Type $A$'') have $2$ ascendants and $q-2$ descendants, while the others (``Type $B$'') have $1$ ascendant and $q-1$ descendants. 
In the figures we denote  vertices of type $A$ by red circles and  vertices of type $B$ by cyan diamonds, while the wingers by white diamonds (according to the denotations in \cite{BNSz}). The vertices which are $n$-edge-long far from the base vertex are in row $n$. 
The general method of preparing the graph is the following: we go along the vertices of the $j^{\text{th}}$ row, according to the type of the elements (winger, $A$, $B$), we draw the appropriate number of edges downwards ($2$, $q-2$, $q-1$, respectively). Neighbour edges of two neighbour vertices of the $j^{\text{th}}$ row meet in the $(j+1)^{\text{th}}$ row, constructing a new  vertex of type $A$. The other descendants of row $j$ have type $B$ in row $j+1$.
In the sequel, $\binomh{n}{k}$ denotes the $k^\text{th}$ element in row $n$, which is either the sum of the values of its two ascendants or the value of its unique ascendant. We note, that the hyperbolic Pascal triangle has the property of vertical symmetry. 

\begin{figure}[h!]
 \centering
  \includegraphics[width=0.99\linewidth]{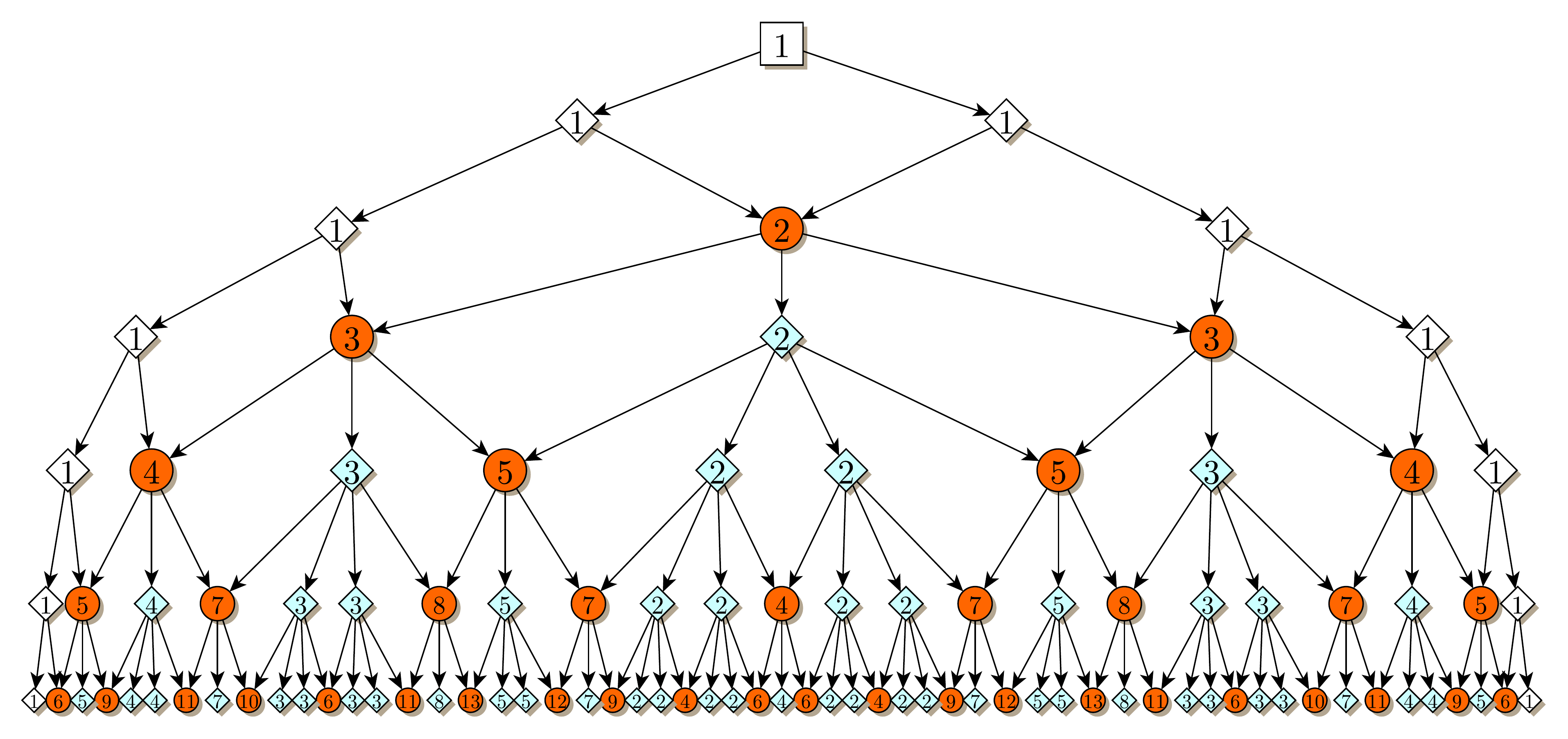}
 \caption{Hyperbolic Pascal triangle linked to $\{4,5\}$ up to row 6}
 \label{fig:Pascal_layer6}
\end{figure}

The 3-dimensional analogue of the original Pascal's triangle is the well-known Pascal's pyramid (or more precisely Pascal's tetrahedron). 
Its levels are triangles and the numbers along the three edges of the $n^{\text{th}}$ level are the numbers of the $n^{\text{th}}$ line of Pascal's triangle. Each number inside in any level is the sum of the three adjacent numbers on the level above (see \cite{ANVI,B, har,Hilton}). 
In the hyperbolic space based on the hyperbolic regular cube mosaic (cubic honeycomb) with Schl\"afli's symbol $\{4,3,5\}$ was defined a hyperbolic Pascal pyramid \mbox{(\hpp\!\!)} as a generalisation of the hyperbolic Pascal triangle (\hpt\!\!) linked to mosaic $\{4,5\}$ and the classical Pascal's pyramid  (\cite{N_pyr}).

The space \H2R is one of the eight simply connected 3-dimensional maximal homogeneous Riemannian geometries (Thurston geometries \cite{MolPSz}, \cite{Thur}). This Seifert fibre space is derived by the direct product of the hyperbolic plane $\mathbf{H}^2$ and the real line $\mathbf{R}$.  For more details see \cite{Szir}. In the following we define the Pascal pyramids in this space based on the so-called $(h^2r)$-cube mosaics similarly to the definition of \hpp (in \cite{N_pyr}). The definition also could be extended to the other regular tiling, but we deal with the prism tilings with square, because they are the most natural generalizations of the original Pascal's triangle and pyramid. This work was suggested by Professor Emil Moln\'{a}r.

\section{Construction of the Pascal pyramid \ppHq}

In the space \H2R we define an infinite number of so-called  $(h^2r)$-cube mosaics.  
We take a hyperbolic plane $\Pi$ as a reference plane and a regular squared mosaic with $\{4,q\}$ ($q\geq5$) on it. Denote dy $d$ the common length of the sides of the squares in this mosaic.  We consider the hyperbolic planes parallel to $\Pi$, where the distance between two consecutive ones is $d$. Let the same mosaic $\{4,q\}$ be defined on all the planes and let the corresponding vertices of the mosaics be on the same Euclidean lines which are perpendicular to the hyperbolic planes. A $(h^2r)$-cube is the convex hull of two corresponding congruent squares on two consecutive hyperbolic mosaics. All the  $(h^2r)$-cubes yield a $(h^2r)$-cube mosaic in the space \H2R based on the regular hyperbolic planar mosaic $\{4,q\}$. 
(In \cite{Szir} the $(h^2r)$-cube mosaics were called prism tilings with squares.) 
Figure \ref{fig:H2R_disk} shows three consecutive hyperbolic planes with mosaic $\{4,5\}$ and some lines perpendicular to these planes. 
Let $V$ be a mosaic vertex on a hyperbolic plane. The vertex figure of $V$ is a double $q$-gon based pyramid, where the vertices are the nearest mosaic vertices to $V$, all vertices are one-edge-long far from $V$. Their base vertices are on the same hyperbolic plane on which $V$ is. 
The vertex figure of $V$ is illustrated in Figure \ref{fig:H2R_disk} (or Figure \ref{fig:H2R_vertex_figures}). 
We mention, that the edges of the vertex figures are not the edges of the mosaic, they are the diameters of its faces.

\begin{figure}[h!]
 \centering
  \includegraphics[width=0.5\linewidth]{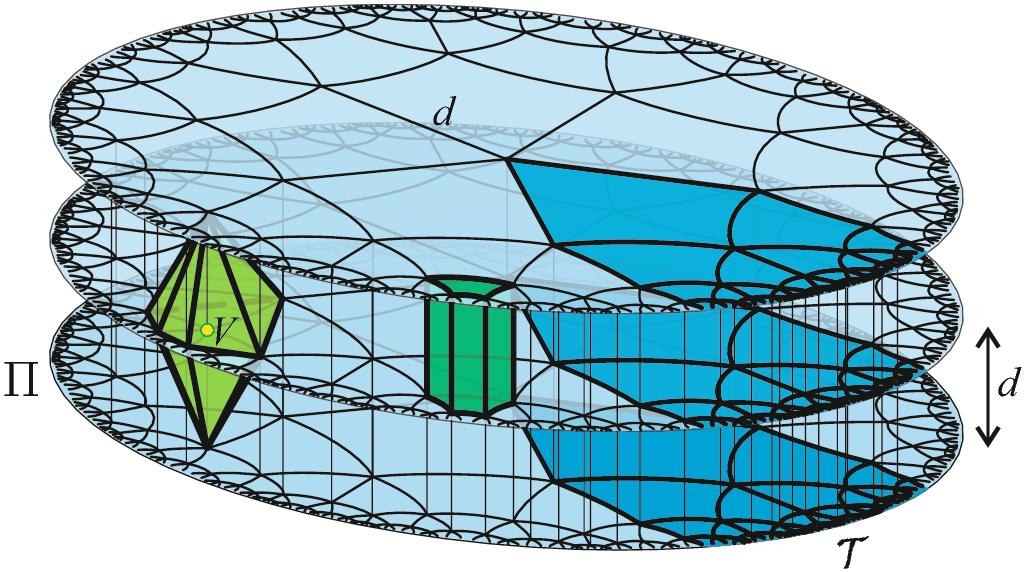}
 \caption{Construction of the $(h^2r)$-cube mosaic based on $\{4,5\}$}
 \label{fig:H2R_disk}
\end{figure}

Take the part $\mathcal{T}$ of the mosaic $\{4,q\}$ on $\Pi$ on which the hyperbolic Pascal triangle \hpt can be defined  (see \cite{BNSz})  and let ${\cal P}$ be the part of the $(h^2r)$-cube mosaic which contains $\mathcal{T}$ and  all its corresponding vertices on other hyperbolic planes that are "above" plane $\Pi$ (the hyperbolic planes which are in the same half space bordered by $\Pi$). (Obviously, ${\cal P}$ contains also the corresponding edges between the vertices.) The shape of this convex part of the mosaic resembles an infinite tetrahedron. This part ${\cal P}$ is darkened in Figure \ref{fig:H2R_disk}.

Let $V_0$ be the base vertex of \hpt on plane $\Pi$. Let ${\cal G}_{\cal P}$ be the digraph directed  according to the growing edge-distance from $V_0$, in which the vertices and edges are the vertices and edges of ${\cal P}$. We label an arbitrary vertex $V$ of ${\cal G}_{\cal P}$  by the number of different shortest paths along the edges of ${\cal P}$ from $V_0$ to $V$. (We mention that all the edges of the mosaic are equivalent.) 
Let the labelled digraph ${\cal G}_{\cal P}$ be the Pascal pyramid (more precisely the Pascal tetrahedron) in space  $\mathbf{H}^2\!\times\!\mathbf{R}$, denoted by \ppHq\!\!. 
Some labelled vertices can be seen in Figure \ref{fig:H2R_disk_Pascal} in case \ppH45\!\!. 

\begin{figure}[h!]
 \centering
 \includegraphics[width=0.7\textwidth]{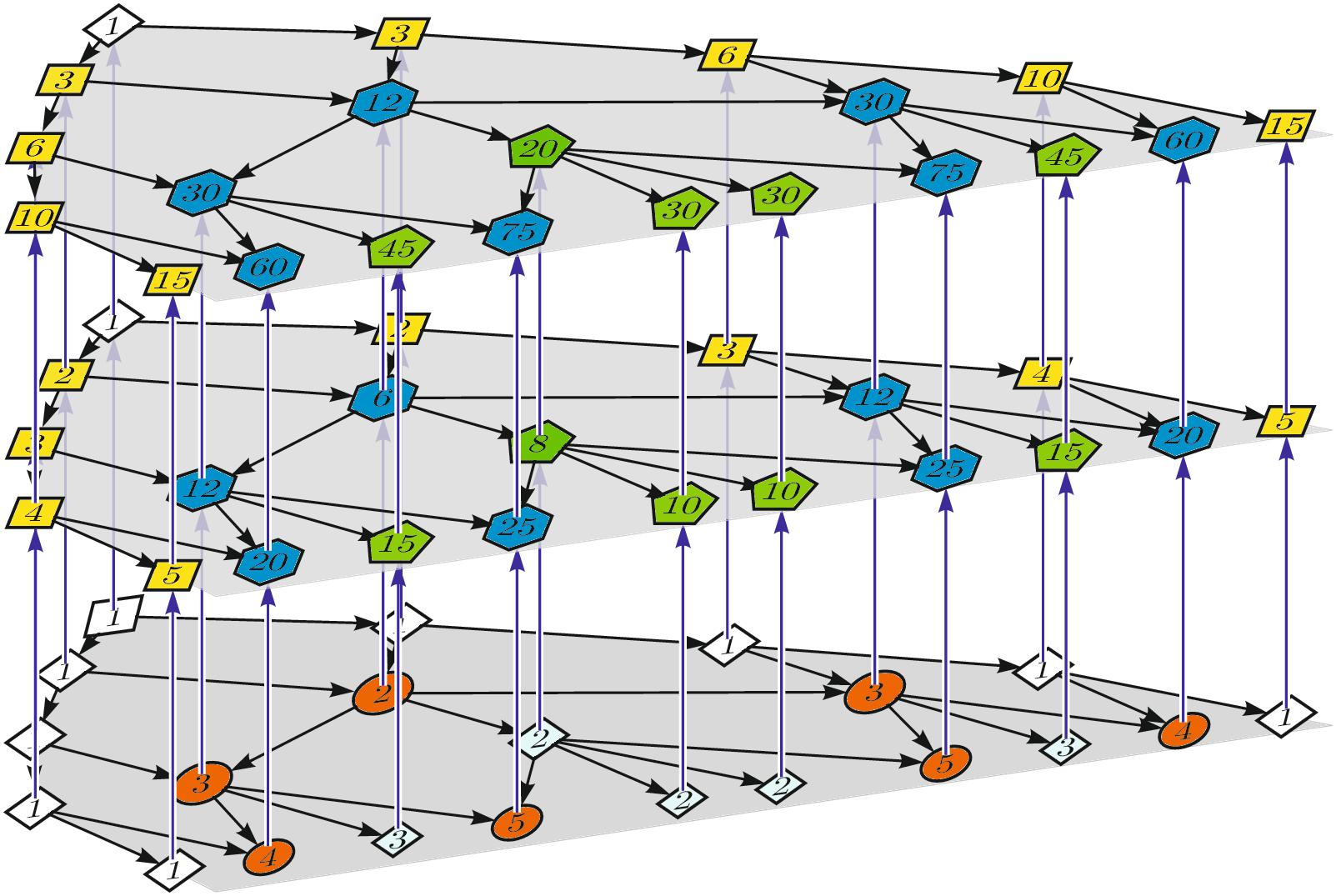}
  \caption{Pascal piramid \ppH45 in space \H2R}
 \label{fig:H2R_disk_Pascal}
\end{figure}

Let level~0 be the vertex $V_0$. Level~$n$ consists of the vertices of \ppHq whose edge-distances from $V_0$ are $n$-edge (the distance of the shortest path along the edges of ${\cal P}$ is $n$).   
It is clear, that one (infinite) face of \ppHq is a \hpt in plane $\Pi$ and the other two faces are Euclidean Pascal's triangles. 
Figure~\ref{fig:pyramid_ppH45} shows the Pascal pyramid \ppH45\! in \H2R  up to level~4. 
Moreover, Figures~\ref{fig:H2Rlayer2to3}--\ref{fig:H2Rlayer5to6} show the growing from a level to the next one in case of some lower levels. The colours and shapes of different types of the vertices are different. (See the definitions later.) The numbers without colouring and shapes refer to vertices in the lower level in each figures. The graphs growing from a level to the new one contain graph-cycles with six nodes, which refer to the convex hulls of the parallel projections of the cubes from the mosaic, where the direction of the projection is not parallel to any edges of the cubes.   
                                                                             
\begin{figure}[h!]
 \centering
 \includegraphics[width=0.6\textwidth]{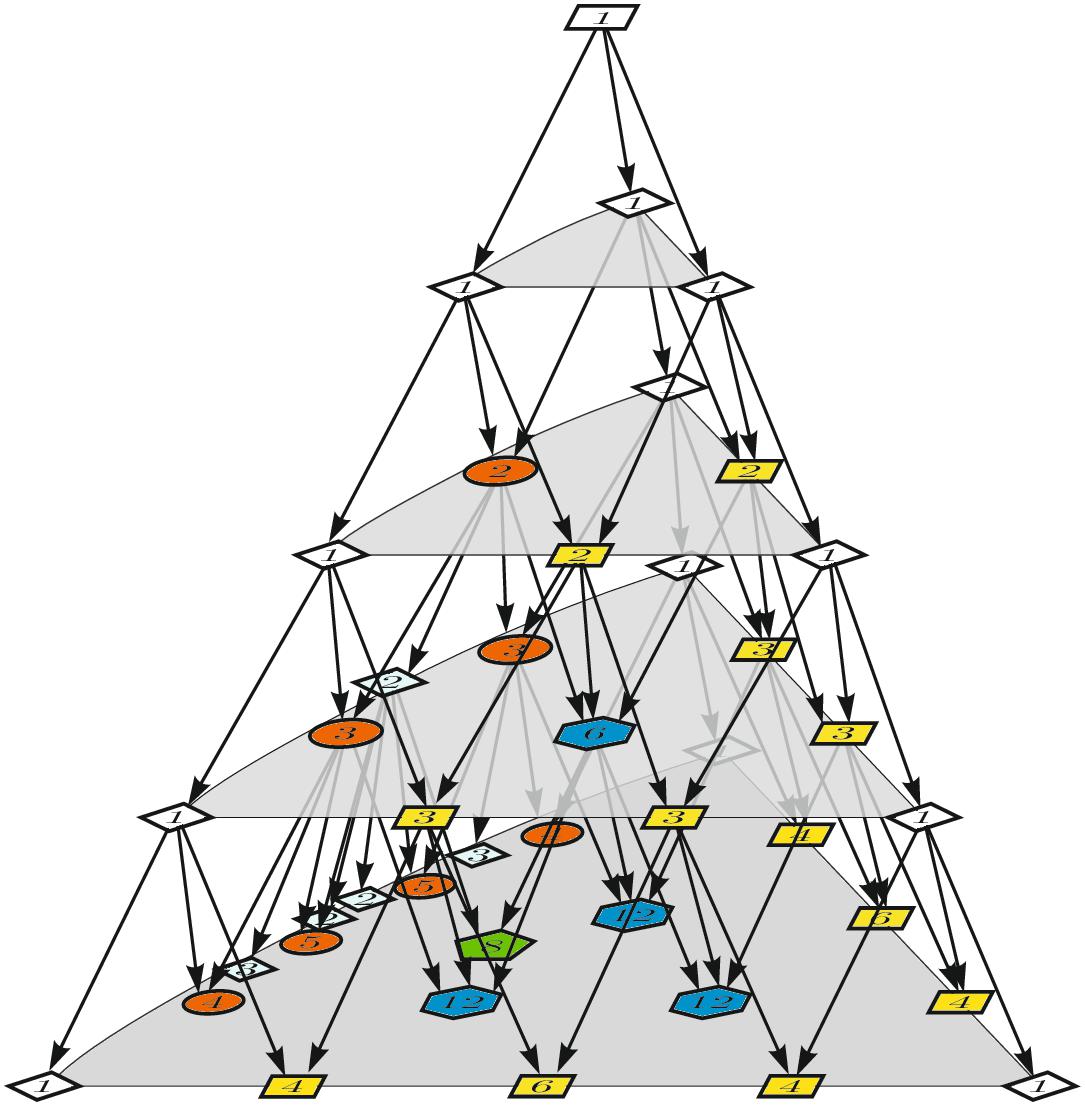}  
 \caption{Pascal pyramid \ppH45}
 \label{fig:pyramid_ppH45}
\end{figure}

\begin{figure}[h!]
 \centering
 \includegraphics[scale=1]{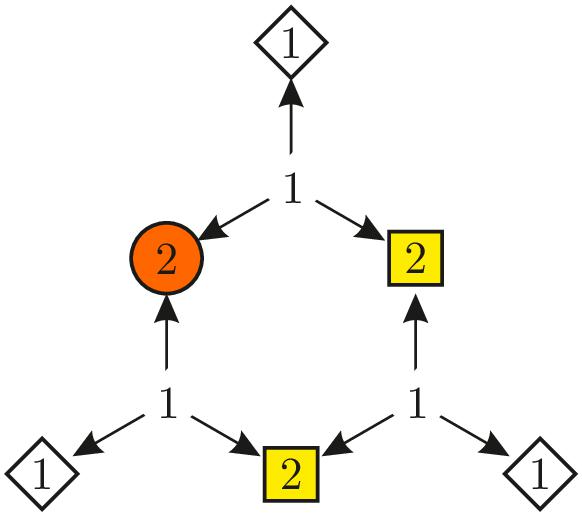}
 \includegraphics[scale=1]{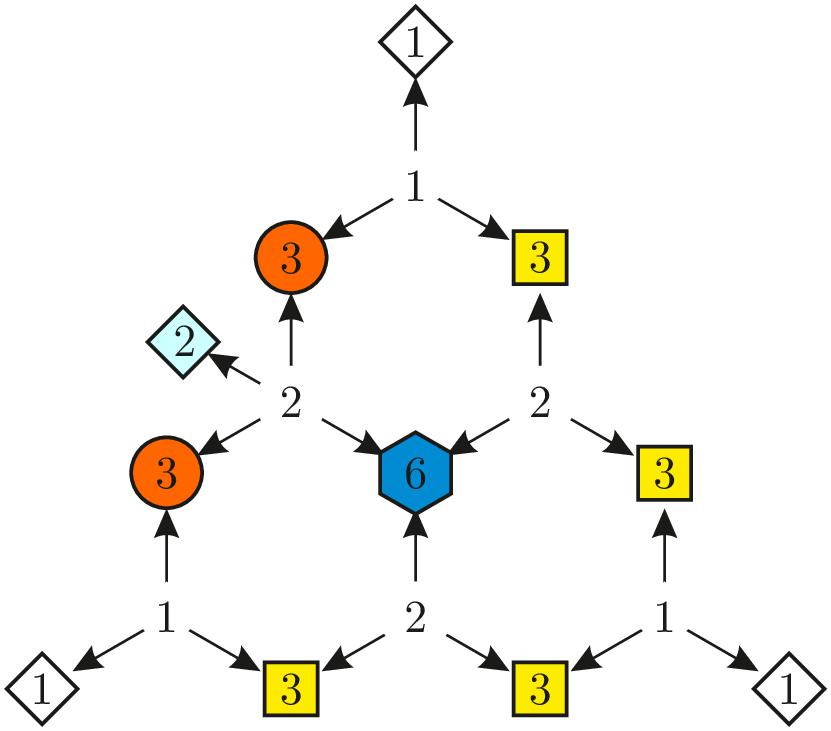}
 \caption{Connection between levels one, two and three in \ppH45}
 \label{fig:H2Rlayer2to3}
\end{figure}

\begin{figure}[ht!]
 \centering
 \includegraphics[scale=1]{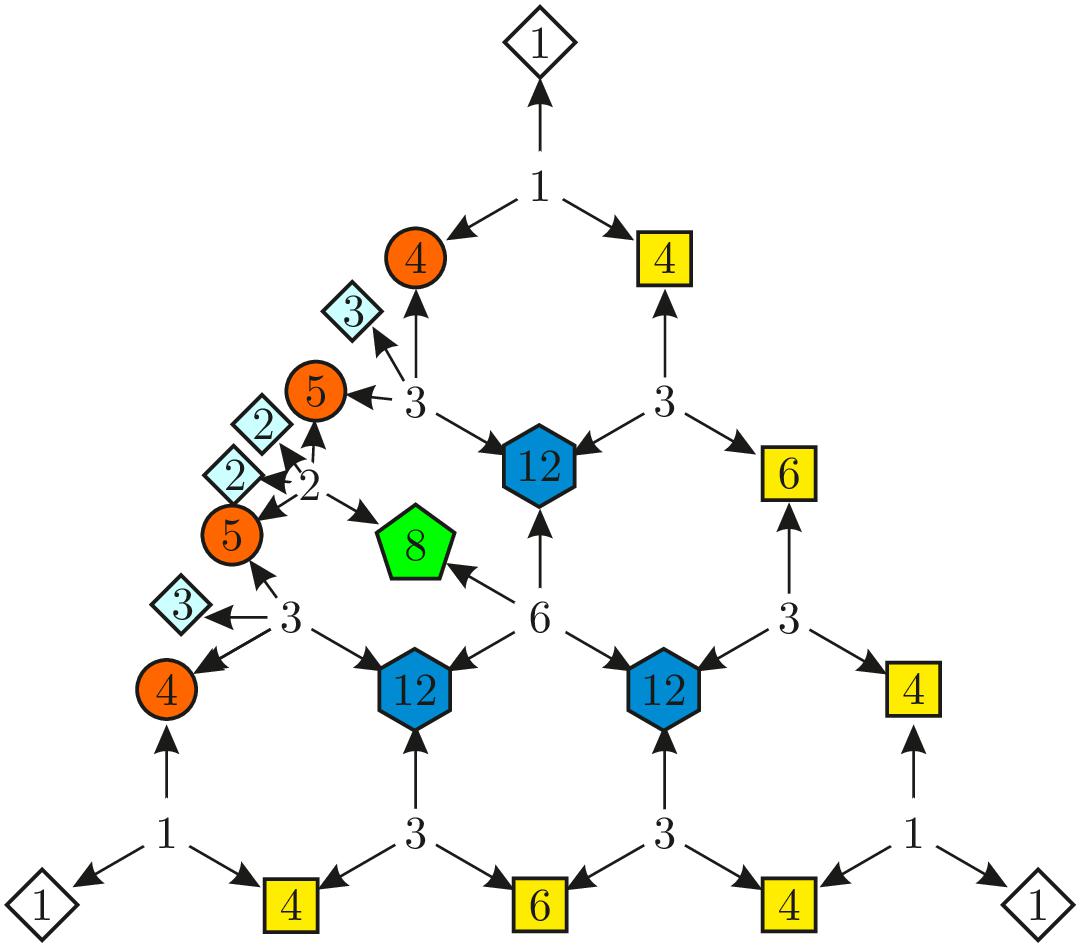}
 \caption{Connection between levels three and four in \ppH45}
 \label{fig:H2Rlayer3to4}
\end{figure}
\begin{figure}[ht!]
 \centering
 \includegraphics[scale=0.8]{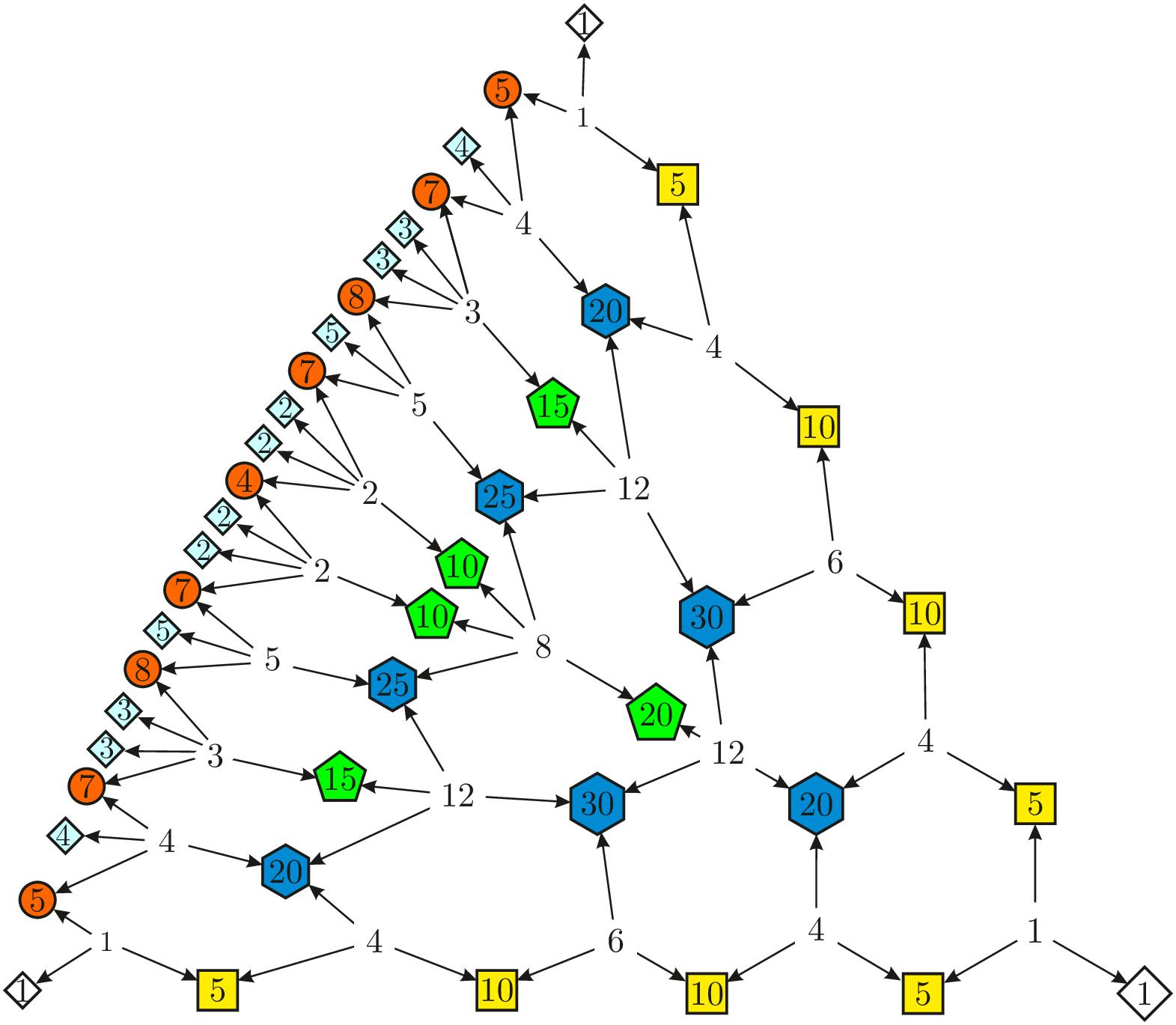}
 \caption{Connection between levels four and five in \ppH45}
 \label{fig:H2Rlayer4to5}
\end{figure}

\begin{figure}[ht!]
 \centering
 \includegraphics[scale=0.8]{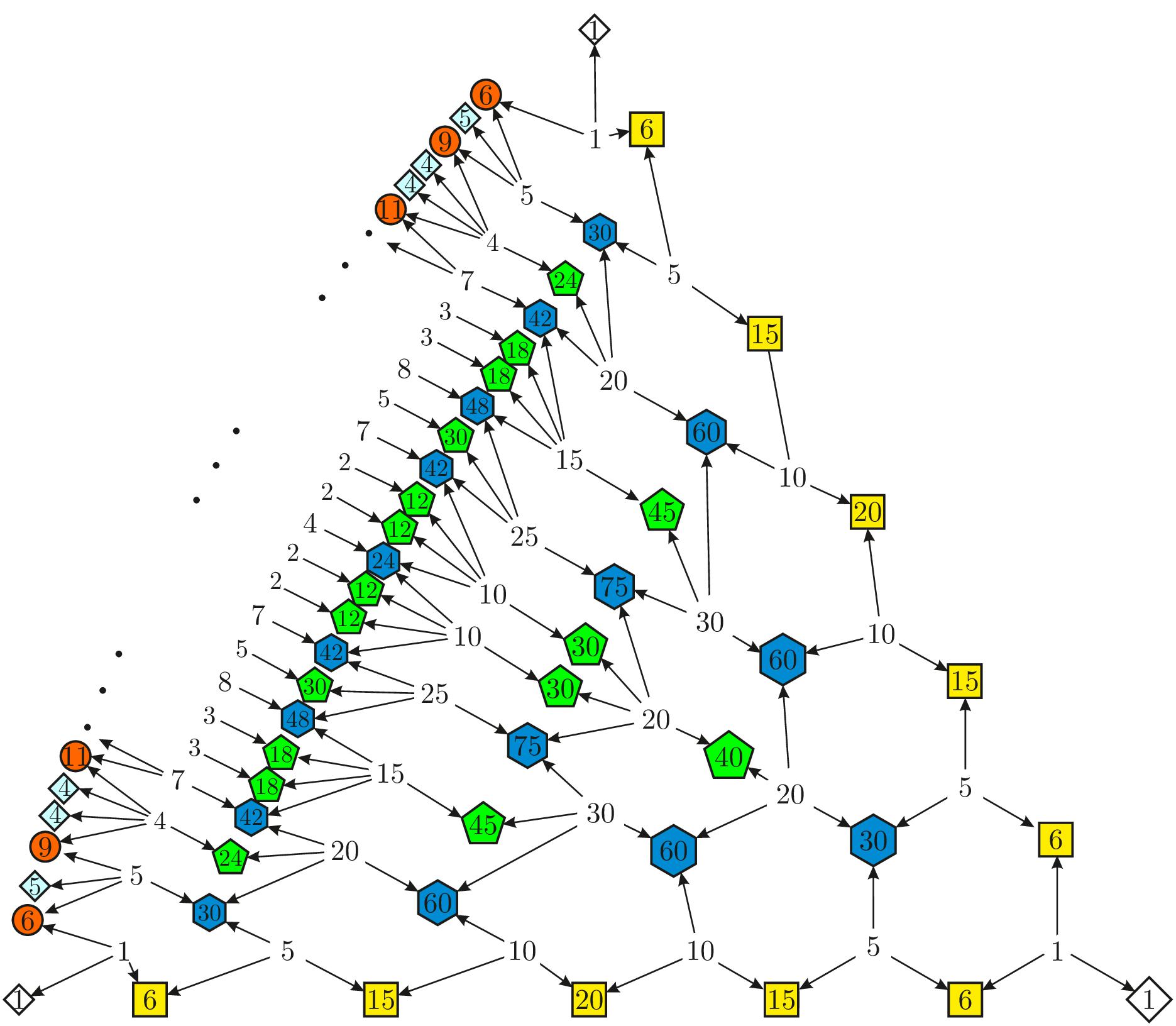}
 \caption{Connection between levels five and six  in \ppH45}
 \label{fig:H2Rlayer5to6}
\end{figure}

In the following we describe the method of the growing of \ppHq and we give the sum of the paths connecting vertex $V_0$ and level $n$.

\section{Growing of \ppHq}

In the classical Pascal's pyramid the number of the elements on  level $n$ is $(n+1)(n+2)/2$ and its growing  from level $n$ to $n+1$ is $n+2$, but in the hyperbolic Pascal pyramid it is more complex (see \cite{N_pyr}).

As one face of \ppHq is the hyperbolic Pascal triangle, then there are three types of vertices $A$, $B$ and $1$ corresponding to the Introduction. The denotations of them are also the same. From all $A$ and $B$  only one edge starts each to the inside of the pyramid, these are the Euclidean edges of the mosaic (see Figure \ref{fig:H2R_disk_Pascal}). The types of the inside vertices of these edges differ from the types $A$ and $B$, let us denote them by type $D$ and type $E$, respectively.
The other two sides of the Pascal pyramid are Euclidean Pascal's triangles, which have two types of vertices, let us denote them  by $C$ and $1$. For a vertex $C$ connects three new vertices in the next level, two vertices $C$ on the side of \ppHq  and one vertex of type $D$ inside the pyramid.  
Sometimes, if it is important, we distinguish the types $1$. If a vertex of type $1$ belongs to \hpt we write it by $1_h$.  

The growing methods of them are illustrated in  Figure~\ref{fig:H2R_graph_side} (compare it with the growing method in \cite{BNSz} and \cite{N_pyr}). In the figures we denote  vertices of type $C$ by yellow squares.   

\begin{figure}[h!]
 \centering
 \includegraphics[scale=0.95]{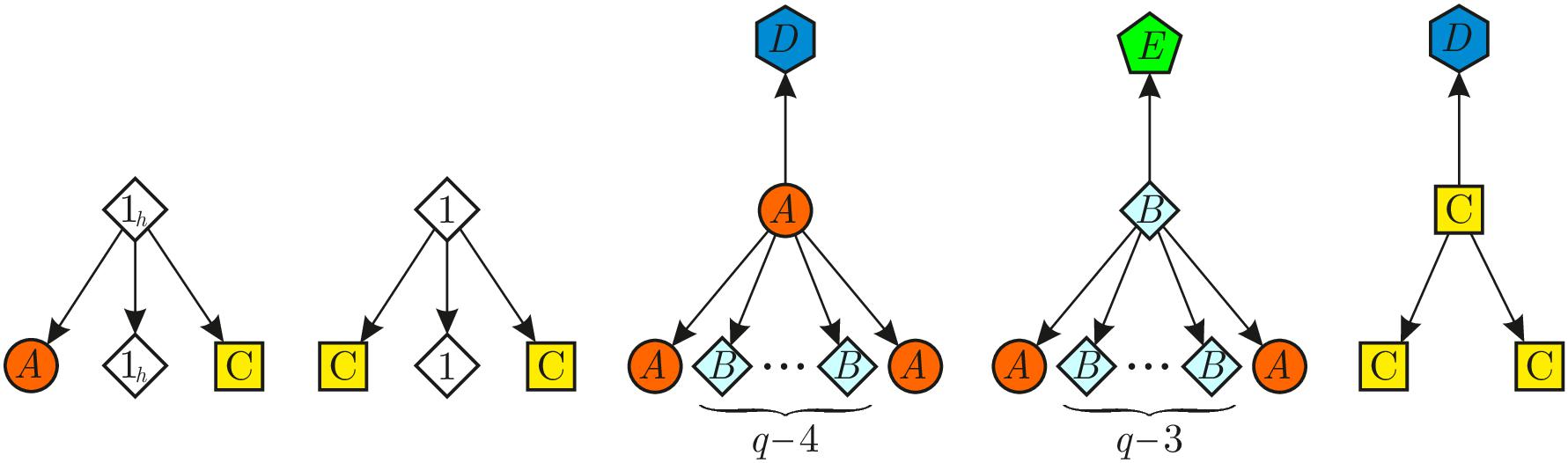}
  \caption{Growing method in case of the faces}
 \label{fig:H2R_graph_side} 
\end{figure}

For the classification and the exact definitions of the inner vertices we examine the vertex figures of the inner vertices. As the structure descends from a hyperbolic plane to the consecutive one, there are two types of the vertex figures.   
During the growing (step from level $i-1$ to level $i$) an arbitrary inner  vertex $V$ on level $i$ can be reached from level $i-1$ with three or two edges. This fact allows us a classification of the inner vertices. Let the type of a vertex on level $i$ be $D$ or $E$, respectively, if it has three or two joining edges to level $i-1$ (as before). Figure~\ref{fig:H2R_vertex_figures} shows the vertex figures of the inner vertices of \ppHq\!\!. Vertices $W_{i-1}$ (small green circles) are on  level $i-1$, $W_i$ and the centres are on level~$i$.  We don't know the types of $W$ (or not important to know).
The other vertices of the double pyramid are on level $i+1$ and the classification of them gives their types. An edge of the double pyramid and its centre $V$ determine a square (a side-face of a $(h^2r)$-cube) from the mosaic. (Recall, that an edge of the vertex figure is a diagonal of a side-face of a $(h^2r)$-cube.)  Since from a vertex of a square we can go to the opposite vertex two ways, 
then a vertex $D$ of the double pyramid, where $D$ and a $W_{i-1}$ are connected by an edge, can be reached with two paths from level $i-1$. 
(For example on the left hand side of Figure~\ref{fig:H2R_vertex_figures}, between a vertex $W_{i-1}$ and $D$ there are the paths $W_{i-1}\!-\!D_i\!-\!D$ and $W_{i-1}\!-\!W_i\!-\!D$.)

\noindent So, the type of the third vertex of the faces on the double pyramid whose other two vertices are $W_{i-1}$ is $D$. The others connect to only one $W_{i-1}$, they can be reached by two ways from level~$i$,   their types are $E$. See Figure~\ref{fig:H2R_vertex_figures}. 
In the figures we denote  vertices of type $D$ by blue hexagons and  vertices of type $E$ by green pentagons. The blue thick directed edges are mosaic-edges between levels $i-1$ and $i$, while the red thin ones are between levels $i$ and $i+1$.

\begin{figure}[h!]
 \centering
 \includegraphics[scale=1]{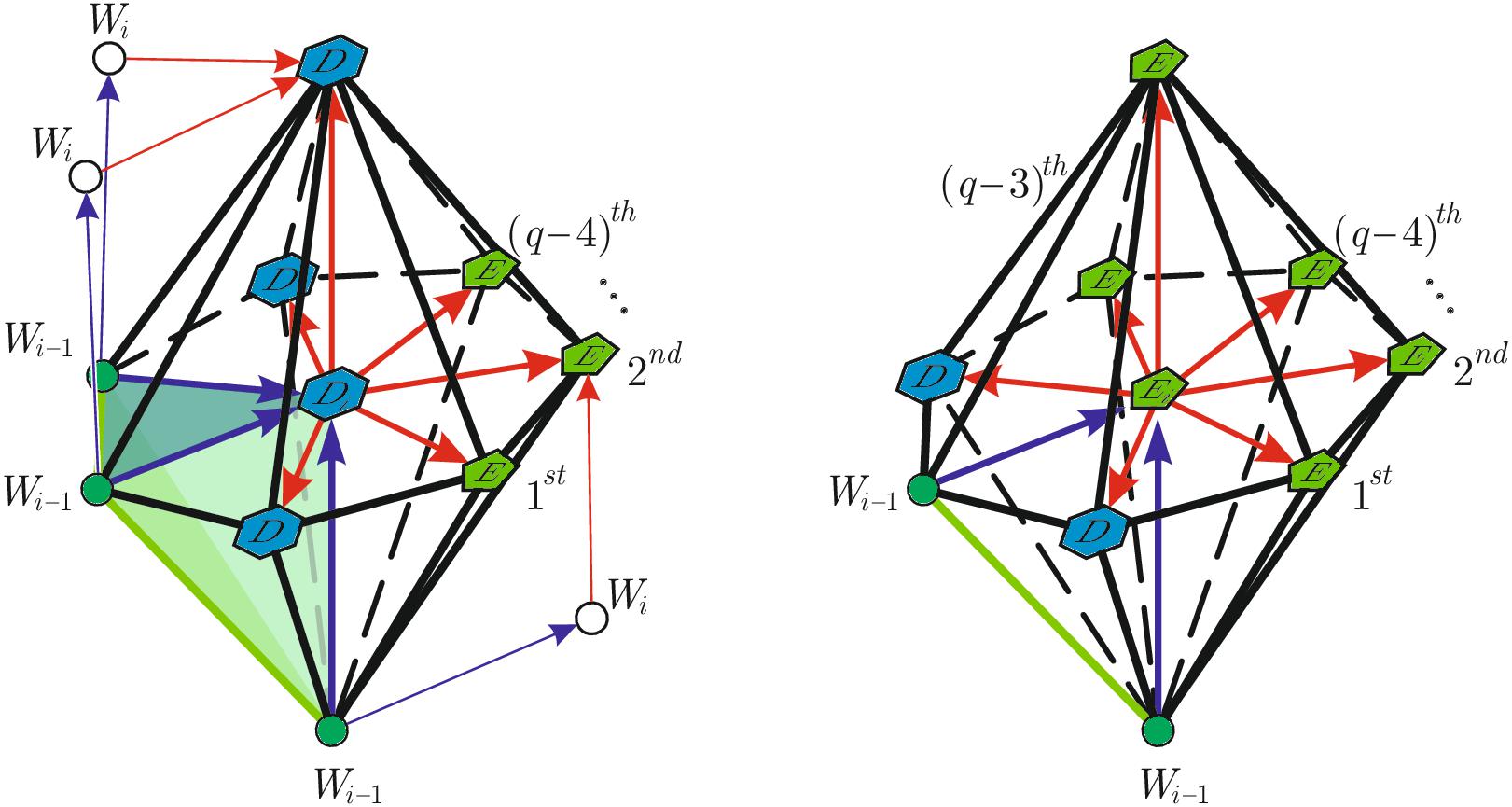}
 \caption{Growing method in case of the inner vertices with vertex figures}
 \label{fig:H2R_vertex_figures}
\end{figure}

In Figure~\ref{fig:H2R_graph_inner} the growing method is presented in case of the inner vertices. These vertices are the centres and some vertices of the double pyramids are presented in Figure~\ref{fig:H2R_vertex_figures}. 

\begin{figure}[h!]
 \centering
 \includegraphics[scale=1]{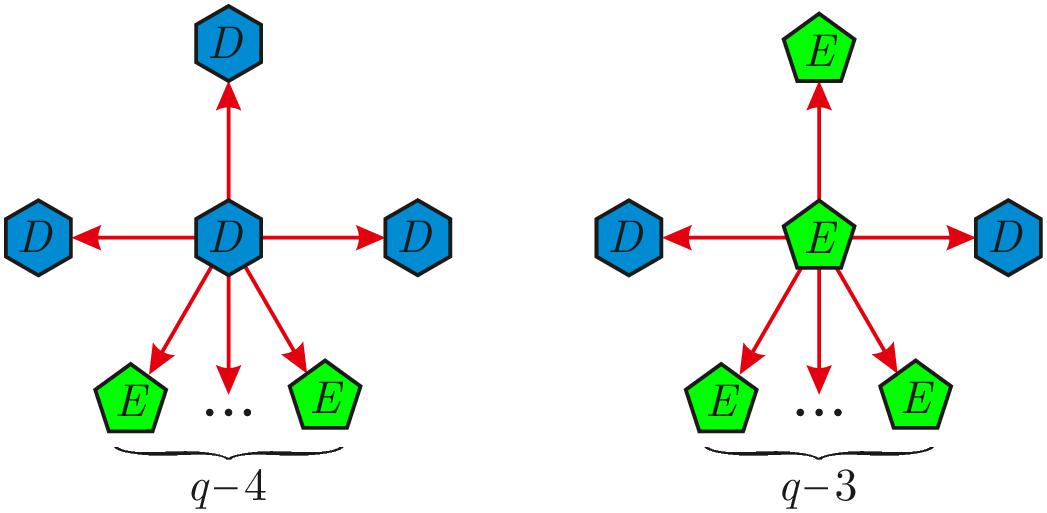}
 \caption{Growing method in case of the inner vertices}
 \label{fig:H2R_graph_inner}
\end{figure}


Finally, we denote the sums of vertices of types $A$, $B$, $C$,  $D$ and  $E$ on level $n$  by   $a_n$, $b_n$, $c_n$, $d_n$ and $e_n$, respectively. 

Summarising the details we prove Theorem~\ref{th:H2R_growing_type}.

\begin{theorem}\label{th:H2R_growing_type}
The growing of the numbers of the different types of the vertices are described by the system of linear inhomogeneous recurrence sequences  $(n\geq1)$
\begin{equation}\label{eq:H2R_seq01}
  \begin{array}{ccl}
a_{n+1}&=& a_n+b_n+1,\\
b_{n+1}&=& (q-4)a_n+(q-3)b_n,\\
c_{n+1}&=& c_n+2,\\
d_{n+1}&=& a_n+d_n,\\
e_{n+1}&=&  b_n+e_n,
  \end{array}
\end{equation}
with zero initial values.
\end{theorem}
We mention that $c_n$ $(n\geq1)$ is an arithmetical sequence and $c_{n}= 2(n-1)$.

\begin{lemma}\label{lem:H2R_de}
For the sequences $d_n$ and $e_n$ $(n\geq1)$ hold 
\begin{equation}
  \begin{array}{ccl}
d_{n}&=& -a_n+\frac{1}{q-4}b_n+(n-1),\\
e_{n}&=& a_n-(n-1).
  \end{array}
\end{equation}
\end{lemma}
\begin{proof}
Obviously, for $n=1$ the equations hold. In case $n>1$  we suppose that $d_{n-1}= -a_{n-1}+\frac{1}{q-4}b_{n-1}+(n-2)$ and $e_{n-1}= a_{n-1}-(n-2)$. Firstly, from  the first, second and fourth rows of \eqref{eq:H2R_seq01} we have 
\begin{equation*}
  \begin{array}{ccl}
a_n+d_n&=&a_{n-1}+b_{n-1}+1+a_{n-1}+d_{n-1}=a_{n-1}+b_{n-1}+1+\frac{1}{q-4}b_{n-1}+(n-2)\\
&=&\frac{q-4}{q-4}a_{n-1}+\frac{q-3}{q-4}b_{n-1}+n-1=\frac{1}{q-4}b_n+(n-1).
  \end{array}
\end{equation*}
Secondly,  from the first and fifth row of \eqref{eq:H2R_seq01} we gain 
$a_n-e_n=a_{n-1}-e_{n-1}+1= (n-2)+1=n-1.$
\end{proof}

Moreover, let $s_n$ $(n\geq1)$ be the number of all the vertices on level $n$, so that $s_0=1$ and
\begin{equation}\label{eq:H2R_sn}
  \begin{array}{ccl}
s_{n}&=& a_n+b_n+c_n+d_n+ e_n+3 \\
  &=& a_n+\frac{q-3}{q-4}b_n+2n+1.
  \end{array}
\end{equation}

Table \ref{table:H2R_typeof_vertices} shows the numbers of the vertices on levels up to 10 in case \ppH45\!\!.  

\begin{table}[!hb]
  \centering \setlength{\tabcolsep}{0.4em}
\begin{tabular}{|c||c|c|c|c|c|c|c|c|c|c|c|}
  \hline
 $n$   &  0  &  1 & 2 & 3 & 4 & 5 & 6  & 7  & 8   & 9   & 10  \\ 
 \hline \hline
 $a_n$ &  0  &  0 & 1 & 2 & 4 & 9 & 22 & 56 & 145 & 378 & 988  \\ \hline
 $b_n$ &  0  &  0 & 0 & 1 & 4 & 12 & 33 & 88 & 232 & 609 & 1596  \\ \hline
 $c_n$ &  0  &  0 & 2 & 4 & 6 & 8 & 10  & 12  & 14   & 16   & 18         \\ \hline
 $d_n$ &  0  &  0 & 0 & 1 & 3 & 7 & 16  & 38  & 94   & 239    & 617         \\ \hline
 $e_n$ &  0  &  0 & 0 & 0 & 1 & 5 & 17  & 50  &  138 & 370   & 979       \\ \hline
 $s_n$ &  1  &  3 & 6 & 11 & 21 & 44 & 101  & 247  &  626 & 1615  & 4201       \\ \hline
 \end{tabular}
\caption{\emph{Number of types of vertices $(n\leq10)$ in case \ppH45}\label{table:H2R_typeof_vertices}}
\end{table}

\begin{theorem}\label{theorem:H2R_numvertex4q}  
The sequences $\{a_n\}$, $\{b_n\}$, $\{c_n\}$, $\{d_n\}$, $\{e_n\}$  and $\{s_n\}$ can be described by the same fourth order linear homogeneous recurrence sequence 
\begin{equation}\label{eq:H2R_recurcde}
x_n=qx_{n-1}+2(1-q)x_{n-2}+qx_{n-3}-x_{n-4} \qquad (n\ge5),\\
\end{equation}
the initial values be can gained from Theorem \ref{th:H2R_growing_type}  (in case \ppH45 from  Table \ref{table:H2R_typeof_vertices}).
The sequences $\{a_n\}$, $\{b_n\}$ can be also described by  
\begin{equation}
x_n=(q-1)x_{n-1}+(1-q)x_{n-2}+x_{n-3} \qquad (n\ge4).\label{eq:H2R_recurab}
\end{equation}
\noindent Moreover, the explicit formulae
\begin{eqnarray*}
a_n&\!\!\!\!\!=&\!\!\!\!\!\left(\frac{2-q}{2}+\frac{D+2}{2D}\sqrt{D}\right)\alpha_1^n+\left(\frac{2-q}{2}-\frac{D+2}{2D}\sqrt{D}\right)\alpha_2^n+1, \quad \\
b_n&\!\!\!\!\!=&\!\!\!\!\!\left(\frac{q-3}{2}+\frac{1-q}{2q}\sqrt{D}\right)\alpha_1^n+\left(\frac{q-3}{2}-\frac{1-q}{2q}\sqrt{D}\right)\alpha_2^n-1, \nonumber\\
d_n&\!\!\!\!\!=&\!\!\!\!\! \left(\frac{q^2-5q+5}{2(q-4)}-\frac{q^2-3q-1}{2D}\sqrt{D}\right)\alpha_1^n+\left(\frac{q^2-5q+5}{2(q-4)}+\frac{q^2-3q-1}{2D}\sqrt{D}\right)\alpha_2^n 
\nonumber \\& & \qquad 
+n-\frac{1}{q-4}+1, \nonumber \\
e_n&\!\!\!\!\!=&\!\!\!\!\! \left(\frac{2-q}{2}+\frac{D+2}{2D}\sqrt{D}\right)\alpha_1^n+ \left(\frac{2-q}{2}-\frac{D+2}{2D}\sqrt{D}\right)\alpha_2^n-n+2, \nonumber \\
s_n&\!\!\!\!\!=&\!\!\!\!\! \left(\frac{q}{2}-\frac{\sqrt{D}}{2D}\right)\alpha_1^n+ \left(\frac{q}{2}+\frac{\sqrt{D}}{2D}\right)\alpha_2^n+2n-\frac{1}{q-4}+
1, \nonumber 
\end{eqnarray*}
are valid, where $D=q(q-4)$, $\alpha_1=(q-2+\sqrt{D})/2$ and $\alpha_2=(q-2-\sqrt{D})/2$.
\end{theorem}

\begin{proof}
The sequences $\{a_n\}$ and $\{b_n\}$ are the same as $\{a_n\}$, $\{b_n\}$ in \cite{BNSz}. So they can be  described by \eqref{eq:H2R_recurab} and their explicit formulae hold. According to Lemma \ref{lem:H2R_de} and  \eqref{eq:H2R_sn}  more explicit formulae are derived by the combination of explicit formulae of $\{a_n\}$ and $\{b_n\}$.

Let us extend \eqref{eq:H2R_recurab} to \eqref{eq:H2R_recurcde} considering the  sequence $\{a_n\}$. Substitute $a_n$ into \eqref{eq:H2R_recurab} and sum $a_n$ and $a_{n-1}$ than we receive implicit form \eqref{eq:H2R_recurcde} for $\{a_n\}$. Similarly, \eqref{eq:H2R_recurcde} is also the implicit form of  $\{b_n\}$. 

Now we prove that \eqref{eq:H2R_recurcde}  holds  for $s_n$ also. 
Multiply the equation \eqref{eq:H2R_sn} by $q$, $2(1-q)$, $q$ and $-1$  where $n=n, n-1, n-2$ and $n-3$, respectively. If we sum them, then we obtain
\begin{multline*}
qs_{n}+2(1-q)s_{n-1}+qs_{n-2}-s_{n-3}=
 qa_{n}+2(1-q)a_{n-1}+qa_{n-2}-a_{n-3}+\\
 \frac{q-3}{q-4}\left(qb_{n}+2(1-q)b_{n-1}+qb_{n-2}-b_{n-3}\right)+ \\
  q(2n+1)+2(1-q)\left(2(n-1)+1\right)+q\left(2(n-2)+1\right)-\left(2(n-3)+1\right)=\\
 a_{n+1}+
  \frac{q-3}{q-4}b_{n+1}+ 2n+3=s_{n+1}.
\end{multline*}
This proves \eqref{eq:H2R_recurcde} for $\{s_n\}$.
 
In case of $\{d_n\}$ and $\{e_n\}$ we can prove  the relation \eqref{eq:H2R_recurcde} similarly to $\{s_n\}$.
\end{proof}

\begin{remark}
If $q=4$, then the sequences give results of the Euclidean Pascal's pyramid. Its all faces are Pascal's triangle, thus $b_n=0$, $e_n=0$, $a_n$ coincide $c_n$, that way the growing equation system  \eqref{eq:H2R_seq01} is just $c_{n+1}=c_n+3$, $d_{n+1}=c_n+d_n$.     
\end{remark}

\begin{remark}
The generating function of the sequence  $s_n$ is given by
\begin{equation*}
\frac{1-(q-3)x-(q-4)x^2}{1-qx+(2q-2)x^2 -qx^3+x^4}.
\end{equation*}
In the case \ppH45\!\! it is 
\begin{equation*}
\frac{1-2x-x^2}{1-5x+8x^2-5x^3+x^4},
\end{equation*} which is not in the OEIS at present.
\end{remark}

\section{Sum of the values on levels in \ppHq}

In this section we determine the sum of the values of the elements on level $n$. 

Denote, respectively, $\hat{a}_{n}$, $\hat{b}_{n}$, $\hat{c}_{n}$, $\hat{d}_{n}$ and $\hat{e}_{n}$ the sums of the values of  vertices of type $A$, $B$, $C$, $D$ and $E$ on level $n$. 
\begin{theorem} \label{th:H2R_recursum}
If $n\geq1$, then
\begin{equation}\label{eq:seq02}
  \begin{array}{ccl}
\hat{a}_{n+1}&=& 2\hat{a}_n+2\hat{b}_n+2,\\
\hat{b}_{n+1}&=& (q-4)\hat{a}_n+(q-3)\hat{b}_n,\\
\hat{c}_{n+1}&=&  2\hat{c}_n+4,\\
\hat{d}_{n+1}&=&  \hat{a}_n+c_n+3\hat{d}_n+2\hat{e}_n,\\
\hat{e}_{n+1}&=& \hat{b}_n+(q-4)\hat{d}_n+(q-2) \hat{e}_n
  \end{array}
\end{equation}
with zero initial values.
\end{theorem}
\begin{proof}
From Figures \ref{fig:H2R_graph_side} and \ref{fig:H2R_graph_inner} the results  can be read directly. For example  all the  vertices of type $A$, $B$ and $1$ on level $n$ generate two vertices of type $A$ on level $n+1$ and it follows from the first equation of \eqref{eq:seq02}.
\end{proof}

Table \ref{table:H2R_sumof_vertices} shows the sum of the values of the vertices on levels up to 10.

Let $\hat{s}_n$ be the sum of the values of all the vertices on level $n$, then $\hat{s}_0=1$ and 
\begin{equation*}
\hat{s}_n = \hat{a}_n+\hat{b}_n+\hat{c}_n+\hat{d}_n+\hat{e}_n+3 \qquad (n\geq1).
\end{equation*}
The value $\hat{s}_n$ also shows the number of paths from $V_0$ to level~$n$.

\begin{table}[!hb]
  \centering \setlength{\tabcolsep}{0.4em}
\begin{tabular}{|c||c|c|c|c|c|c|c|c|c|c|c|}
  \hline
 $n$         &  0 & 1 & 2 & 3 & 4 & 5 & 6  & 7  & 8   & 9   & 10  \\ 
 \hline \hline
 $\hat{a}_n$ &  0 & 0 & 2 & 6 & 18 & 58 & 194 & 658 & 2242 & 7642 & 26114   \\ \hline
 $b_n$       &  0 & 0 & 0 & 2 & 10 & 38 & 134 & 462 & 1582 & 5406 & 18462  \\ \hline
 $\hat{c}_n$ &  0 & 0 & 4 & 12& 28 & 60 & 124 & 252  & 508   & 1020   & 2044         \\ \hline
 $d_n$       &  0 & 0 & 0 & 6 & 36 & 170& 768 & 3458  & 15596   & 70314    & 316296         \\ \hline
 $\hat{e}_n$ &  0 & 0 & 0 & 0 & 8  & 70 & 418 &  2156 & 10388   & 48342   &  220746  \\ \hline
 $\hat{s}_n$ &  1 & 3 & 9 & 29& 103& 399& 1641& 6989  &  30319 & 132735  & 583665       \\ \hline
 \end{tabular}
\caption{\emph{Sum of values of types of vertices $(n\leq10)$ in case \ppH45}\label{table:H2R_sumof_vertices}}
\end{table}

\begin{theorem}
The sequences $\{\hat{a}_n\}$, $\{\hat{b}_n\}$, $\{\hat{c}_n\}$, $\{\hat{d}_n\}$, $\{\hat{e}_n\}$  and $\{\hat{s}_n\}$ can be described by the same sixth order linear homogeneous recurrence sequence 
\begin{multline}\label{eq:H2R_recurdehat}
x_n= (2q+3)x_{n-1}+(-q^2-7q-5)x_{n-2}+(4q^2+10q+9)x_{n-3}+\\
(-5q^2-13q-10)x_{n-4}+(2q^2+12q+12)x_{n-5}+(-4q-8)x_{n-6}
\end{multline}
the initial values are from the equation system \eqref{th:H2R_recursum}.
The sequences $\{\hat{a}_n\}$, $\{\hat{b}_n\}$ can be also described by  
\begin{equation}\label{eq:H2R_recurabhat}
x_n=qx_{n-1}-(q+1)x_{n-2}+2x_{n-3} \qquad (n\ge3).
\end{equation}
\noindent Moreover, for sequences $\{\hat{c}_n\}$, $\{\hat{s}_n\}$
\begin{equation}\label{eq:H2R_recurchat}
\hat{c}_n=3\hat{c}_{n-1}-2\hat{c}_{n-2} \qquad (n\ge2),
\end{equation}
\begin{equation}\label{eq:H2R_recurshat}
\hat{s}_n=(q+3)\hat{s}_{n-1}-(3q+4)\hat{s}_{n-2}+(2q+4)\hat{s}_{n-3} \qquad (n\ge3)
\end{equation}
and the explicit formula of $\hat{s}_n$
\begin{equation}
\hat{s}_n= \left(-\frac12+(q-1)\frac{\sqrt{D}}{2D}\right)\alpha_1^n+ \left(-\frac12-(q-1)\frac{\sqrt{D}}{2D}\right)\alpha_2^n+2\cdot2^n
\end{equation}
is valid, where $D=q^2-2q-7$, $\alpha_1= \frac12\left(1+q+\sqrt{D}\right)$ and $\alpha_2= \frac12\left(1+q-\sqrt{D}\right)$.
\end{theorem}

We do not give the implicit formulae for all sequences, because generally, they are complicated, but  they can be calculated easily with computer.

\begin{proof}
Let $\hat{v}_n=3$ $(n\geq1)$ be constant sequences and $\hat{v}_0=1$. The value $\hat{v}_n$  gives the sum of the number of  vertices of type ``1" on  level $n$. Substitute  $2=2\hat{v}_n/3$ and $4=4\hat{v}_n/3$ into the first and third equations of \eqref{th:H2R_recursum} and complete it with equation $\hat{v}_{n+1}=\hat{v}_n$. 
Than we have the system of linear homogeneous recurrence sequences $(n\geq1)$
\begin{equation}\label{eq:H2R_seq02v}
  \begin{array}{ccl}
\hat{a}_{n+1}&=& 2\hat{a}_n+2\hat{b}_n+\frac23v_n,\\
\hat{b}_{n+1}&=& (q-4)\hat{a}_n+(q-3)\hat{b}_n,\\
\hat{c}_{n+1}&=&  2\hat{c}_n+\frac43v_n,\\
\hat{d}_{n+1}&=&  \hat{a}_n+c_n+3\hat{d}_n+2\hat{e}_n,\\
\hat{e}_{n+1}&=& \hat{b}_n+(q-4)\hat{d}_n+(q-2) \hat{e}_n,\\
\hat{v}_{n+1}&=& \hat{v}_n  
\end{array}
\end{equation}
and  
\begin{equation*}\label{eq:snv}
\hat{s}_n = \hat{a}_n+\hat{b}_n+\hat{c}_n+\hat{d}_n+\hat{e}_n+\hat{v}_n\qquad (n\geq0).
\end{equation*}

The shorter form of the linear homogeneous recurrence sequences \eqref{eq:H2R_seq02v} is 
\begin{eqnarray}
 \vec{u}_{i+1}&=&\vec{M}\vec{u}_i,\label{eq:H2R_am}
\end{eqnarray}
where $\vec{u}_j=[\hat{a}_j\enskip \hat{b}_j\enskip \hat{c}_j\enskip \hat{d}_j\enskip \hat{e}_j\enskip \hat{v}_j\enskip]^T$ and 
$$\vec{M}=\begin{pmatrix}
2   & 2   & 0   & 0   & 0   & \frac23 \\ 
q-4   & q-3   & 0   & 0   & 0   & 0 \\ 
0 & 0   & 1   & 0 & 0   & \frac43 \\ 
1   & 0 & 1 & 3   & 2 & 0 \\ 
0   & 1  & 0 & q-4   & q-2     & 0 \\ 
0   & 0   & 0   & 0   & 0   & 1  
\end{pmatrix}.$$
The characteristic polynomial of $\vec{M}$ is 
\begin{equation}
p_6(x)=\left( x-1 \right)  \left( x-2 \right)  \left( x^2-(q+1)x+q+2  \right)  \left(x^2+(1-q)x+2 \right)  \label{eq:H2R_karMs}
\end{equation}
and according to Theorem 3 from \cite{N_pyr} the equation \eqref{eq:H2R_karMs} is the characteristic equation of all the sequences $\hat{r}_{i+1}=\vec{y}^T\vec{u}_i$, where $ \vec{y}=[y_1 \enskip y_2 \enskip \dots \enskip y_6]^T$. Thus  \eqref{eq:H2R_karMs} is the characteristic equation of $\hat{s}_n$, $\hat{a}_n$, \ldots, $\hat{e}_n$ with $\vec{y}=(1,1,1,1,1,1)$, $\vec{y}=(1,0,0,0,0,0)$, \ldots, $\vec{y}=(0,0,0,0,1,0)$, respectively. 

The equation \eqref{eq:H2R_karMs} implies the six ordered linear recurrence sequence \eqref{eq:H2R_recurdehat} for the considered sequences, but for the lower order implicit formulae we have to examine the first elements of the sequences.  One  can easily check, that the polynomial  
\begin{eqnarray} 
 p_{\hat{s}}(x)=\left( x-2 \right)  \left( x^2-(q+1)x+q+2  \right) =x^3-(q+3)x^2+(3q+4)x-(2q+4), \label{eq:H2R_karM4s}
\end{eqnarray}
 moreover in case of $\hat{a}_i$  and $\hat{b}_i$  (see \cite{BNSz}) 
\begin{eqnarray*} 
 p_3(x)=(x-1)\left(x^2+(1-q)x+2 \right)=x^3-qx^2+(1+q)x-2, \label{eq:H2R_karM3}
\end{eqnarray*} 
in case of $\hat{c}_i$   the 
\begin{eqnarray*} 
 p_2(x)=(x-1)(x-2)=x^2-3x+2. \label{eq:H2R_karM2}
\end{eqnarray*}
These polynomials imply the recurrence relations \eqref{eq:H2R_recurdehat}--\eqref{eq:H2R_recurshat}, respectively. 

The roots of polynomial \eqref{eq:H2R_karM4s} are $\alpha_1$, $\alpha_2$, $\alpha_3=2$  and  
$$\hat{s}_n=\beta_1\cdot \alpha_1^n+\beta_2\cdot \alpha^n_2+\beta_3\cdot 2^n \qquad (n\geq1)$$  
provide the explicit formulae (see \cite{sho}). Solutions of  the linear equation system  from $n=1,2,3$ determine the exact values of $\beta$'s.   
\end{proof}

\begin{remark}
In the case \ppH45\!\! the growing ratio of values is $\lim_{n \rightarrow  \infty } {\hat{s}_{n+1}}/{\hat{s}_{n}}=\alpha_1\!\approx\!4.414$, where $\alpha_1$ is the largest eigenvalue of matrix $\vec{M}$. Recall, that these 
growing ratios are $3$ and $\approx\!10.351$ in the Euclidean and hyperbolic cases, respectively (\cite{N_pyr}).    
\end{remark}

\begin{remark}
The generating function of the sequence $\hat{s}_n$ is given by 
\begin{equation*}
\frac{1-qx+4x^2}{1-(q+3)x+(3q+4)x^2 -(2q+4)x^3}.
\end{equation*}
In the case \ppH45\!\! it is
\begin{equation*}
\frac{1-5x+4x^2}{1-8x+19x^2-14x^3}, 
\end{equation*} which is not in the OEIS at present.
\end{remark}


\begin{thebibliography}{2015}

\bibitem{AhBel} {\sc M. Ahmia, H. Belbachir},  {\it Preserving log-convexity for generalized Pascal triangles}, { The Electronic Journal of Combinatorics}, \textbf{19}(2) (2012), $\#$P16.

\bibitem{ANVI} {\sc G. Anatriello, G. Vincenzi}, \textit{Tribonacci-like sequences and generalized Pascal's pyramids}, Internat. J. Math. Ed. Sci. Tech. \textbf{45} (2014), 1220--1232.

\bibitem{Barry} {\sc P. Barry},  {\it On integer-sequence-based constructions of generalized Pascal triangles}, { Journal of Integer Sequences}, Vol. 9 (2006), Article 06.2.4.

\bibitem{BNSz}{\sc H. Belbachir, L. N\'emeth, L. Szalay}, {\em Hyperbolic Pascal triangles}, Appl. Math. Comp., {\bf 273}(2016), 453--464.

\bibitem{BSz} {\sc H. Belbachir, L. Szalay.}, {\em On the arithmetic triangles}, \u{S}iauliai Math.~Sem., {\bf 9} (17) (2014), 15--26. 

\bibitem{B}{\sc B. A. Bondarenko}, {\em Generalized Pascal triangles and pyramids, their fractals, graphs, and applications}. Translated from the
Russian by {\sc B. A. Bollinger}, (English) Santa Clara, CA: {\em The Fibonacci Association}, vii, 253 p. (1993). www.fq.math.ca/pascal.html

\bibitem{C} {\sc H. S. M. Coxeter}, {\em Regular honeycombs in hyperbolic space}, Proc.~Int.~Congress Math., Amsterdam, Vol. III.(1954), 155--169.

\bibitem{har}{\sc J. M. Harris,J. L. Hirst, M. J.  Mossinghoff}, {\em Combinatorics and Graph Theory}, Springer, 2008.

\bibitem{Hilton} {\sc P. Hilton, J.  Pedersen}, {Mathematics, Models, and Magz, Part I: Patterns in Pascal's Triangle and Tetrahedron}, {\it Mathematics Magazine}, 85(2), (2012), 97-109.

\bibitem{N_pyr}	{\sc L. N\'emeth},  {\it On the hyperbolic Pascal pyramid}, { Beitr. Algebra Geom.} (2016) 57:913-927.

\bibitem{NSz_alter}  {\sc L. N\'emeth, L. Szalay}, {\it Alternating sums in hyperbolic Pascal triangles}, { Miskolc Mathematical Notes}, (accepted).

\bibitem{NSz_recu}    {\sc L. N\'emeth, L. Szalay},  {\it Recurrence sequences in the hyperbolic Pascal triangle corresponding to the regular mosaic $\{4,5\}$}, {\ Annales Mathematicae et Informaticae}, (accepted).



\bibitem{sho}{\sc T. N. Shorey, R. Tijdeman}, {\em Exponential diophantine equation}, Cambridge University Press, 1986, p. 33.

\bibitem{SV} {\sc S. Siani, G.  Vincenzi},  \textit{Fibonacci-like sequences and generalized Pascal's triangles},  Internat. J. Math. Ed. Sci. Tech. 45 (2014), 609--614.

\bibitem{Szir} {\sc J. Szirmai}, {\em Geodesic ball packings in  space for generalized Coxeter space groups},   
Mathematical Communications,  17/1(2012), 151--170.


\bibitem{MolPSz} {\sc E. Moln\'ar, I.  Prok, J. Szirmai}, {\em Classification of tile-transitive 3-simplex tilings and
	their realizations in homogeneous spaces}, in: {\em Non-Euclidean Geometries}, J\'anos Bolyai
Memorial Volume, (A. Prekopa and E. Moln\'ar, Eds.), Mathematics and Its Applications 581, Springer, (2006), 321--363.

\bibitem{Thur}  {\sc W. P. Thurston}, {\em Three-Dimensional Geometry and Topology}, (S. Levy, Ed.), Princeton University Press, Princeton, New Jersey, 1997.
 



\end{thebibliography}
\end{document}